\documentclass[10pt,letterpaper]{amsart}
\usepackage{amsmath, amsthm, amssymb, amsfonts}
\usepackage{mathrsfs}

\newtheorem{thm}{Theorem}
\newtheorem{cor}[thm]{Corollary}
\newtheorem{lem}[thm]{Lemma}
\newtheorem{prop}[thm]{Proposition}

\newtheorem*{thm*}{Theorem}
\newtheorem*{cor*}{Corollary}
\newtheorem*{lem*}{Lemma}
\newtheorem*{prop*}{Proposition}
 
\newtheorem*{ex*}{Exercise} 

\theoremstyle{definition}
\newtheorem{defn}[thm]{Definition}
\newtheorem*{defn*}{Definition}

\newtheorem*{prob*}{Problem}

\theoremstyle{definition}

\theoremstyle{definition}

\newtheorem{examp}[thm]{Example}
\newtheorem*{example*}{Example}

\theoremstyle{remark}

\newtheorem*{conj*}{Conjecture}

\newcommand{\mc}[1]{\mathcal{#1}}

\newcommand{\mbf}[1]{\mathbf{#1}}

\newcommand{\Z}{{\mathbb Z}}

\newcommand{\N}{{\mathbb N}}

\newcommand{\nin}{\not\in}

\newcommand{\empset}{\{\}}
\renewcommand{\today}{\number \day \space \ifcase \month \or January\or%
  February\or March\or April\or May\or June\or July\or August\or%
  September\or October\or November\or December\fi \space \number \year}

\DeclareMathOperator{\sdepth}{sdepth}
\DeclareMathOperator{\depth}{depth}

\newcommand{\set}[1]{\ensuremath{\left\{ #1 \right\}}}

\usepackage{mathptmx,graphicx,color,enumerate}
\usepackage[abs]{overpic}
\usepackage{subcaption}
\usepackage{hyperref}
\DeclareMathAlphabet{\mathcal}{OMS}{cmsy}{m}{n}
\newcommand{\SR}{Stanley-Reisner}
\begin{document}
\title{Combinatorial Reductions for the Stanley Depth of $I$ and $S/I$}
\author{Mitchel T.\ Keller}
\address{Department of Mathematics\\Washington and Lee University\\204
W Washington Street\\Lexington, VA 24450}
\email{kellermt@wlu.edu}
\author{Stephen J.\ Young}
\address{Pacific Northwest National Laboratory\\Richland, WA 99352}
\email{stephen.young@pnnl.gov}
\thanks{The first author was partially supported in this research by a Marshall
  Sherfield Fellowship and a Robert Price
  Research Award. The second author was partially supported in this
  research by an AMS-Simons Travel Grant.}
\date{27 August 2017}
\thanks{\textit{PNNL Information Release:} PNNL-SA-121188}
\keywords{Stanley depth; monomial ideal; subset lattice; Boolean lattice}
\subjclass[2010]{06A07, 05E40, 13C13}
\maketitle
\begin{abstract}
We develop combinatorial tools to study the relationship between the
Stanley depth of a monomial ideal $I$ and the Stanley depth of its
compliment, $S/I$.  Using these results we are able to prove that if
$S$ is a polynomial ring with at most 5 indeterminates and $I$ is a
square-free monomial ideal, then the Stanley depth of $S/I$ is strictly
larger than the Stanley depth of $I$.  Using a computer search, we are
able to extend this strict inequality up to polynomial rings with at
most 7 indeterminates.  This partially answers questions asked by Propescu
and Qureshi as well as Herzog.
\end{abstract}

\section{Introduction}

In \cite{stanley:linear-dioph}, Richard P.\ Stanley introduced what
has come to be known to be the Stanley depth of a particular class of
modules. Following Herzog's survey article
\cite{herzog:sdepth-survey}, we give the following restricted
definition that suffices for our purposes.

\begin{defn}
  Let $K$ be a field and $S = K[x_1,\dots,x_n]$ the polynomial ring
  over $K$. Let $M$ be a $\Z^n$-graded $S$-module, $m\in M$
  homogenous, and $Z\subset \set{x_1,\dots,x_n}$. We call the
  $K[Z]$-submodule $mK[Z]$ of $M$ a \emph{Stanley space} of $M$ if
  $mK[Z]$ is a free $K[Z]$-submodule of $M$. In this case, we call
  $|Z|$ the \emph{dimension} of $mK[Z]$.

  A \emph{Stanley decomposition} $\mc{D}$ of $M$ is a decomposition of
  $M$ as a direct sum of $\Z^n$-graded $K$-vector spaces
  \[\mc{D}\colon M = \bigoplus_{j=1}^r m_j K[Z_j],\]
  where each $m_j K[Z_j]$ is a Stanley space of $M$. Given a Stanley
  decomposition $\mathcal{D}$, we define its \emph{Stanley depth} to
  be $\sdepth(\mc{D}) = \min\set{|Z_j|\colon j=1,\dots,r}$. The
 \emph{ Stanley depth} of $M$ is then $\sdepth(M) =
 \max_\mc{D}\sdepth(\mc{D})$, where the maximum is taken over all
 Stanley decompositions of $M$.
\end{defn}

In this work we focus on the relationship between $\sdepth{I}$ and
$\sdepth{S/I}$ for an arbitrary monomial ideal $I$.  In
\cite{rauf:depth-sdepth}, Rauf showed that if $I$ is a complete
intersection monomial ideal, then $\sdepth(I) > \sdepth(S/I)$.
Subsequently, Herzog's survey article \cite{herzog:sdepth-survey}
stated the weaker inequality $\sdepth(I)\geq \sdepth(S/I)$ (for
arbitrary monomial ideals) as Conjecture 64, noting that the
inequality is strict in all known cases. The strict version of the
inequality was conjectured by Popescu and Qureshi in
\cite{popescu:compute-sdepth}. They gave their motivation in
Proposition 5.2, which states that this inequality, when combined with
a proof of a conjecture on cyclic $S$ modules, would yield a proof of
a conjecture of Stanley's for all monomial ideals. (A more general
version of Conjecture 64 in \cite{herzog:sdepth-survey} is presented
there as Question 63, which involves the relationship between a module
and its syzygy module. That question has been answered in the negative
by Ichim et al.\ in \cite{ichim:compute-sdepth}.)

The conjecture of Stanley dates back to \cite{stanley:linear-dioph},
where he conjectured that $\sdepth(M)\geq \depth(M)$ for all
$S$-modules $M$. This conjecture has been resolved affirmatively in a
number of cases, but recently Duval et al.\ provided a counterexample
in \cite{duval:sdepth-counterex}. In particular, they disproved an
earlier conjecture given independently by Stanley
\cite{stanley:balanced-cm} and Garsia \cite{garsia:comb-methods}. That
conjecture proposed that every Cohen-Macaulay simplicial complex is
partitionable. By a result of Herzog, Jahan, and Yassemi
\cite{herzog:partition-complex}, this conjecture is equivalent to
Stanley's conjecture when $M$ is $S/I$ for the \SR{} ideal $I$ of a
Cohen-Macaulay simplicial complex.  However, the construction of Duval
et al.\ gives a quotient ideal as the counterexample to Stanley's
conjecture, and thus it remains possible that
the conjecture holds for all monomial ideals. In fact, a
recent conjecture of Katth\"an \cite{katthan:betti-posets} would imply
Stanley's conjecture for monomial ideals as well in addition to
implying that Stanley's conjecture was ``almost'' right for quotients
of monomial ideals (in that we would have $\sdepth S/I \geq \depth S/I - 1$).

Our primary tool in investigating the relationship between
$\sdepth{I}$ and $\sdepth{S/I}$ will be the recent work of Herzog,
Vladoiu, and Zheng \cite{herzog:sdepth-mon-ideal} which allows us to
focus on partitions of partially ordered sets instead of algebraic decompositions.
Specifically,
they showed that if the module $M$ is a monomial ideal $I$ of $S$ or a
quotient $I/J$ of monomial ideals of $S$, it is possible to determine
the Stanley depth of $M$ precisely by considering special partitions
of a poset naturally associated to the monomial ideal (or quotient).

Let $J\subseteq I\subseteq S$ be monomial ideals. Suppose that $I$ is
generated by the monomials
$x^{\mbf{a}_1},x^{\mbf{a}_2},\dots,x^{\mbf{a}_r}$ and $J$ is generated
by $x^{\mbf{b}_1},x^{\mbf{b}_2},\dots,x^{\mbf{b}_s}$. Here the
$\mbf{a}_i$ and $\mbf{b}_j$ are ordered $n$-tuples of nonnegative
integers and $x^{\mbf{c}} = x_1^{\mbf{c}(1)}x_2^{\mbf{c}(2)}\cdots
x_n^{\mbf{c}(n)}$ with $\mbf{c}(i)$ the $i$\textsuperscript{th} entry
of $\mbf{c}$. Notice that the usual componentwise partial order on
$\N^n$ is algebraically meaningful here, as $\mbf{a}\leq\mbf{b}$ if
and only if $x^\mbf{a}$ divides $x^\mbf{b}$. To define a finite
subposet of $\N^n$ associated with $I/J$, we find a $\mbf{g}\in \N^n$
such that $\mbf{a}_i\leq \mbf{g}$ for all $i$ and $\mbf{b}_j\leq
\mbf{g}$ for all $j$. (Such a $\mbf{g}$ clearly exists, as the
componentwise maximum of the $\mbf{a}_i$ and $\mbf{b}_j$ satisfies the
inequalities.) Define a subposet $P_{I/J}^\mbf{g}$ of $\N^n$
associated to $I/J$ by taking as the ground set of $P_{I/J}$ all
$\mbf{c}\in\N^n$ such that
\begin{enumerate}
\item $\mbf{c}\leq \mbf{g}$,
\item $\mbf{c}\geq \mbf{a}_i$ for some $i$, and
\item $\mbf{c}\not\geq \mbf{b}_j$ for all $j$.
\end{enumerate}
We generally will take $\mbf{g}$ to be as described above and then
write $P_{I/J}$ for $P^\mbf{g}_{I/J}$.

The idea of Herzog et al.\ in \cite{herzog:sdepth-mon-ideal}  was to
partition the poset $P_{I/J}$ into intervals $[\mbf{a},\mbf{b}] =
\set{\mbf{c}\in\N^n\colon \mbf{a}\leq \mbf{c}\leq\mbf{b}}$ and
make the following definitions.

\begin{defn}
  Let $\mc{P}$ be a partition of $P^\mbf{g}_{I/J}$ into intervals. We
  call such a partition a \emph{Stanley partition}. For
  $\mbf{c}\in P^\mbf{g}_{I/J}$, the number of coordinates in which
  $\mbf{c}$ is equal to $\mbf{g}$ is denoted by $\alpha(\mbf{c}) = |\set{i\in
    [n]\colon \mbf{c}(i) = \mbf{g}(i)}|$. The
  \emph{Stanley depth} of $\mc{P}$ is $\sdepth(\mc{P}) =
  \min_{[\mbf{a},\mbf{b}]\in\mc{P}} \alpha(\mbf{b})$, and the \emph{Stanley
    depth} of $P^\mbf{g}_{I/J}$ is
  $\sdepth(P^\mbf{g}_{I/J})=\max_{\mc{P}} \sdepth(\mc{P})$, where the
  maximum is taken over all partitions of $P^\mbf{g}_{S/I}$ into
  intervals. If $\sdepth(\mc{P}) = \sdepth(P_{I/J}^\mbf{g})$, we say
  that $\mc{P}$ is \emph{optimal}.
\end{defn}

Herzog et al.\ showed in \cite{herzog:sdepth-mon-ideal} that $\sdepth(I/J)
= \sdepth(P^\mbf{g}_{I/J})$, providing a mechanism (albeit not
terribly efficient) for computing the Stanley depth of a monomial
ideal (or quotient of monomial ideals). This result generated a flurry
of new activity, including the calculation of the Stanley depth of the
maximal ideal of $K[x_1,\dots,x_n]$ 
\cite{k:interval-partition-sdepth}, results on the Stanley depth of
squarefree Veronese ideals of $K[x_1,\dots,x_n]$ for some degrees 
\cite{ge:sdepth-sqfree-veronese,k:sdepth-sqfree-veronese}, results on
the Stanley depth of complete intersection monomial ideals
\cite{k:sdepth-squarefree,okazaki:sdepth,shen:sdepth-monomial-ideal}, and CoCoA
implementations of the algorithm
\cite{ichim:hdepth-algorithm,rinaldo:sdepth-algorithm}.

In this paper, we investigate the conjectured inequality using the
poset partition approach. We prove the conjecture (with strict
inequality) for $n=3$. We then restrict our arguments to the case
where $I$ is squarefree, since in this context the posets $P_I$ and
$P_{S/I}$ can be viewed as complementary subposets of the subset
lattice $\mbf{2}^n$. In addition to being easier to read, arguments
and results focusing on squarefree monomial ideals were given
additional emphasis when Ichim et al.\ showed in
\cite{ichim:sdepth-polarization} that
\[\sdepth I/J-\depth I/J=\sdepth I^p/J^p-\depth I^p/J^p,\]
where $J\subseteq I\subseteq S$ are ideals and $I^p$ denotes
polarization. 

To facilitate our arguments for the squarefree cases, Sections~\ref{sec:lemmas},
\ref{sec:comb-crit}, and \ref{sec:splits} establish a
number of ancillary results focused mainly on the
structure of a minimal counterexample. For $n=4$ and $n=5$, we prove
the strict conjecture in the case where $I$ is squarefree. We also
describe computational methods that have verified that if $I$ is
squarefree, the strict conjecture also holds for $n=6$ and
$n=7$. Therefore, the principal result of this paper is:

\begin{thm*}
  Let $K$ be a field and $n\leq 7$ an integer. If $I$ is a squarefree
  monomial ideal of $S=K[x_1,\dots,x_n]$, then $\sdepth I>\sdepth S/I$.
\end{thm*}


\subsection*{Notation and Terminology} If $P$ is a poset, we call the
set \[D(A) = \set{x\in P\mid x< a\text{ for some }a\in A}\]
the \emph{down set} of $A$ in $P$. The \emph{closed down set} of $A$
is $D[A] = D(A)\cup A$. Dually, the \emph{up set} of $A$ in $P$ is
$U(A) = \set{x\in P\mid x > a\text{ for some }a\in A}$ and
$U[A] = U(A)\cup A$ is the \emph{closed up set} of $A$ in $P$. If $S$
is a set, we denote by $\mbf{2}^S$ the lattice of all subsets of
$S$. In the case of $[n] = \set{1,2,\dots,n}$, we will often write
$\mbf{2}^n$ for $\mbf{2}^{[n]}$. 

If $A$ is an antichain in $\mbf{2}^{n}$, then $D[A]$ can be viewed a
simplicial complex $\Delta$ with $A$ as its set of facets. If
$I_\Delta$ is the \SR{} ideal of a simplicial complex
$\Delta$, then $P_{S/I_\Delta}$ is the same as $\Delta$. Similarly,
for a squarefree monomial ideal $I$, $P_{S/I}$ is the \SR{}
complex of $I$, and the antichain of maximal elements of $P_{S/I}$ is
the set of facets of the simplicial complex. If all the facets of a
simplicial complex have the same dimension, then the complex is said
to be \emph{pure}. The dimension of a simplicial complex is one less
than the maximum size of a facet. Because the approach
taken here focuses on the poset perspective, we will prefer that
terminology and viewpoint for much of this paper, although we will
freely pass between the different terminologies.

\section{The $n=3$ case}

Before developing our collection of results that provide additional
power in the squarefree case, we address the general ideal case for
$n=3$.

\begin{thm}\label{thm:n-3}
  If $n=3$ and $I\subseteq S$ is a monomial ideal, then 
  $\sdepth I > \sdepth S/I$.
\end{thm}

\begin{proof}
  Let $I=(x^{\mbf{a}_1},\dots,x^{\mbf{a}_t})$ and fix $\mbf{g}\in\N^n$
  such that $\mbf{g}\geq \mbf{a}_i$ for all $i$. Then $P_I$ and
  $P_{S/I}$ are disjoint and their union is the subposet of $\N^n$
  containing all elements less than or equal to $\mbf{g}$. Let $M$ be
  the set of maximal elements of $P_{S/I}$. For
  $\mbf{c}\in P_I\cup P_{S/I} = D[\mbf{g}]$, let
  \[\alpha(\mbf{c}) = |\set{i\in [n] \mid \mbf{c}(i) = \mbf{g}(i)}|.\]
  If $C\subseteq P_I\cup P_{S/I}$, let
  $\alpha(C) = \min_{\mbf{c}\in C}\alpha(\mbf{c})$.

  Since $\sdepth(P_{S/I})\leq \alpha(M)$, it suffices to show that
  $\sdepth(P_I)> \alpha(M)$. Our proof will be by induction on
  $m=|M|$. The base case is $m=1$.  We take the single element of $M$
  to be an ordered triple $\mbf{b}$ and consider cases determined by
  $\alpha(\mbf{b})$.
  \begin{enumerate}
    \setcounter{enumi}{-1}
  \item Suppose $\alpha(\mbf{b}) =0$. Partition $P_I$ into a collection of
    one-dimensional intervals $[(i,j,k),(i,j,\mbf{g}(3))]$ for $0\leq i\leq
    \mbf{g}(1)$, $0\leq j\leq \mbf{g}(2)$, and $k$ as small as possible so that
    $(i,j,k)\in P_I$. Each interval has at least one coordinate of its
    upper bound in agreement with $\mbf{g}$, so $\sdepth P_I\geq 1$.
  \item Suppose $\alpha(\mbf{b}) = 1$. Without loss of generality,
    assume that $\mbf{b}(1)=\mbf{g}(1)$ and
    $\mbf{b}(2)\geq \mbf{b}(3)$. We partition $P_I$ into two
    intervals. The first is \[I_1=[(0,\mbf{b}(2)+1,0),(\mbf{g}(1),\mbf{g}(2),\mbf{b}(3))],\]
    which lies completely inside $P_I$ because every element has
    second coordinate too large to belong to $P_{S/I}$. The second is
    $I_2=[(0,0,\mbf{b}(3)+1),\mbf{g}]$. Again, $I_2$ lies inside $P_I$ as the
    third coordinate of every point is too large to belong to
    $P_{S/I}$. Furthermore, $I_1\cap I_2 = \empset$ because
    $\mbf{c}_2(3)>\mbf{c}_1(3)$ for every $\mbf{c}_1\in
    I_1,\mbf{c}_2\in I_2$. To show that $I_1\cup I_2 = P_I$, note that
    we must ensure that any $\mbf{c}$ with $\mbf{c}(i) > \mbf{b}(i)$
    for some $i$ is in one of the intervals. If $\mbf{c}(3) >
    \mbf{b}(3)$, then $\mbf{c}\in I_2$. If $\mbf{c}(2) > \mbf{b}(2)$
    and $\mbf{c}(3)\leq \mbf{b}(3)$,
    then $\mbf{c}\in I_1$. Since $\mbf{b}(1) = \mbf{g}(1)$, there are
    no elements of $P_i$ with first coordinate greater than
    $\mbf{b}(1)$. Thus
    $\sdepth P_I\geq 2$.
  \item If $\alpha(\mbf{b}) = 2$, then we may simply partition $P_I$ into a
    single interval of the form $[(0,0,\mbf{b}(3)+1),\mbf{g}]$, assuming without
    loss of generality that $\mbf{b}(1)=\mbf{g}(1)$ and $\mbf{b}(2)=\mbf{g}(2)$.
  \item The case $\alpha(\mbf{b})=3$ is absurd, as it would imply that $P_I$
    is empty.
  \end{enumerate}

  Before proceeding to the inductive step, we must eliminate the case
  where $\alpha(M) = 2$. Note that in this case, $|M|\leq 3$ and for
  each $i=1,2,3$, there is at most one element $\mbf{z}_i$ of $M$ that
  disagrees with $\mbf{g}$ in coordinate $i$. We can partition $P_I$
  into the single interval $[\mbf{c},\mbf{g}]$ with (for $i=1,2,3$)
  $\mbf{c}(i) = \mbf{z}_i(i)+1$ if $\mbf{z}_i$ exists and
  $\mbf{c}(i)=0$ otherwise.

  Now suppose that for some $m\geq 1$, if $|M|\leq m$, then
  $\sdepth P_I > \alpha(M)$. We consider the situation where $M$ has
  $m+1$ elements and $\alpha(M) < 2$. Suppose there is $\mbf{b}\in M$
  with $\alpha(\mbf{b})=2$.  Without loss of generality,
  $\mbf{b}=(\mbf{g}(1),\mbf{g}(2),\mbf{b}(3))$. The interval
  $[(0,0,0),\mbf{b}]$ may be used in the partition of $P_{S/I}$, so
  delete it from $P_{S/I}$ and reduce the third coordinate of every
  remaining point by $\mbf{b}(3)+1$ to leave a poset $Q$. Here we
  may consider the set $M'$ consisting of the elements of $M-\mbf{b}$
  shifted by reducing their third coordinate by $\mbf{b}(3) +1$. Now
  we apply induction to determine that
  $\sdepth P_I > \alpha(M')=\alpha(M)$ as required.

  If $\alpha(M)=0$, partition $P_I$ using intervals
  $[(i,j,k),(i,j,g(3))]$ for all $i,j$, choosing $k$ as small as
  possible so that $(i,j,k)\in P_I$. This partition is guaranteed to
  work since $P_I$ is an up set. Hence $\sdepth P_I \geq 1>\alpha(M)$.

  We may now assume that $\alpha(M)= 1$. We have already taken care of
  the case where there is $\mbf{b}\in M$ with $\alpha(\mbf{b}) = 2$,
  so $\alpha(\mbf{b}) = 1$ for all $\mbf{b}\in M$. For $\mbf{b}\in M$,
  define
  $\gamma(\mbf{b})=\max_{i\colon \mbf{b}(i)\neq \mbf{g}(i)}
  \mbf{b}(i)$.
  Now choose $\mbf{b}_0$ so that $\gamma(\mbf{b}_0)$ is
  maximum. Without loss of generality, $\mbf{b}_0(1) = \mbf{g}(1)$ and
  $\mbf{b}_0(2)=\gamma(\mbf{b}_0)$.

  We wish to find an interval that can be used in a partition of $P_I$
  that will allow us to decrease the size of $M$.  To do this,
  consider the interval $J=[\mbf{s},\mbf{t}]$ where $\mbf{s}=(c_1,c_2,0)$ and
  $\mbf{t}=(\mbf{g}(1),\mbf{g}(2),c_3)$ with
  \[c_1 = 1+\max_{\substack{\mbf{b}\in M\colon\\ \mbf{b}(1)\neq
      \mbf{g}(1)\\\mbf{b}(2)=\mbf{g}(2)}} \mbf{b}(1), \quad c_2 =
  \mbf{b}_0(2)+1,\quad\text{and }\quad c_3 = \min_{\mbf{b}\in M}
  \mbf{b}(3).\]
  If there is no $\mbf{b}\in M$ with $\mbf{b}(1)\neq \mbf{g}(1)$ and
  $\mbf{b}(2)=\mbf{g}(2)$, then let $c_1=0$. We must argue that
  $\mbf{s}\nin P_{S/I} = D[M]$. Our choice of $\mbf{s}(2)$ ensures
  that $\mbf{s}(2)$ is not less than any $\mbf{m}\in M$ unless
  $\mbf{m}(2)=\mbf{g}(2)$. However, for such an $\mbf{m}$,
  $\mbf{s}(1)$ is too large to lie in $D[M]$ unless there is
  $\mbf{m}\in M$ with $\alpha(\mbf{m})=2$, a case we have already
  eliminated.

  In order to be able to apply the induction hypothesis, we must show
  that $D[\mbf{t}] - J\subseteq D[M]$. To do this, take
  $\mbf{x}\in D[\mbf{t}] - J$. Since $\mbf{x}\nin U[\mbf{s}]$, we must
  have $\mbf{x}(1) < c_1$ or $\mbf{x}(2) < c_2$. If
  $\mbf{x}(2) < c_2$, then $\mbf{x}\leq \mbf{b}_0$ since
  $\mbf{b}_0(1) = \mbf{g}(1)\geq \mbf{x}(1)$ and $\mbf{x}(3)\leq
  c_3\leq \mbf{b}_0(3)$. Thus $\mbf{x}\in D[M]$.
  Suppose that $\mbf{x}(2)\geq c_2$. Now we must have
  $\mbf{x}(1) < c_1$.  Since $\mbf{x}(1)\geq 0$, this requires that
  $c_1\neq 0$, and thus there is $\mbf{b}\in M$ with
  $\mbf{b}(1)\neq \mbf{g}(1)$ and $\mbf{b}(2)=\mbf{g}(2)$. Take such a
  $\mbf{b}$ with $\mbf{b}(1)=c_1-1$. This $\mbf{b}$ is greater than or
  equal to $\mbf{x}$, implying $\mbf{x}\in D[M]$.

  For an illustration of the situation at this point, see
  Figure~\ref{fig:split-cube}. Notice now that any $\mbf{b}\in M$ with $\mbf{b}(3) = c_3$ (and
  there is at least one such $\mbf{b}$) must satisfy
  $\mbf{b}\in D[\mbf{t}]$. Thus,
  $D[\mbf{g}] - D[\mbf{t}] = [(0,0,c_3+1),\mbf{g}]$ contains at most
  $m$ elements of $M$. Call this set of maximal elements
  $M'\subseteq M$. We may therefore, by induction, find an interval
  partition $\mc{P'}_I$ of
  $P_I - D[\mbf{t}]$ with $\sdepth(P_I - D[\mbf{t}]) >
  1=\alpha(M')$. Adding $J$ to $\mc{P}'_I$ yields a partition
  $\mc{P}_I$ of $P_I$ witnessing $\sdepth P_I = 2 > 1 = \alpha(M)$
  since $\alpha(\mbf{t}) = 2$. (Note that if $|M'| = 0$, then
  $D[\mbf{g}] - D[\mbf{t}]$ becomes the only interval in the partition $\mc{P}'_I$.)
\end{proof}
  \begin{figure}
    \centering
\begin{overpic}[width=0.5\textwidth]{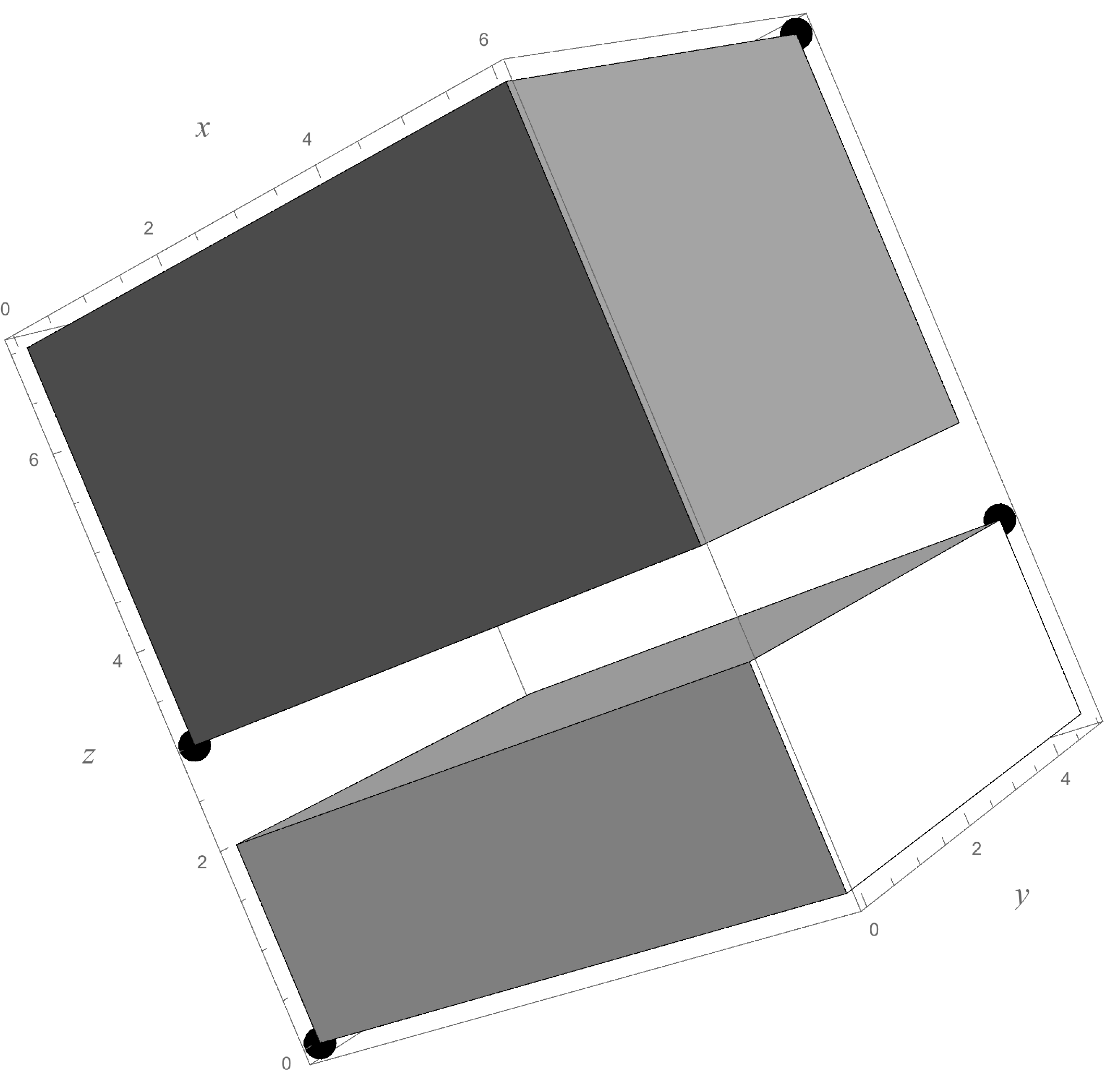}
\put(15,0){$(0,0,0)$}
\put(130,180){$\mbf{g}$}
\put(167,92){$\mbf{t}$}
\put(-20,48){$(0,0,c_3+1)$}
\end{overpic}
\caption{$D[\mbf{g}]$ and $D[\mbf{t}]$}
\label{fig:split-cube}
\end{figure}

\section{Properties of a minimal counterexample}
\label{sec:lemmas}

For the remainder of the paper, we will focus on the case where $I$ is
a squarefree monomial ideal of $S$. Combinatorially, this provides an
advantage in that a squarefree monomial $x^\mbf{a}$ has every entry in
$\mbf{a}$ equal to either $0$ or $1$. Thus, we may identify $\mbf{a}$
with $A=\set{i\in [n]\colon \mbf{a}(i) = 1}$. In doing so, $P_I$
becomes an up set in the subset lattice $\mbf{2}^n$, and $P_{S/I}$ is
the complementary down set of $P_I$ in $\mbf{2}^n$. (From an algebraic
perspective, $P_{S/I}$ can be viewed as the \SR{} simplicial
complex of $I$.) We are then able to use more straightforward language
and ideas connected to the subset lattice than the coordinate-based
argument of the previous section. We begin with an elementary property
of the subset lattice.

\begin{prop}
  There is an interval partition of $\mbf{2}^n-\empset$ in which the
  minimal element of every interval is a singleton.\label{prop:singleton-part}
\end{prop}

\begin{proof}
  Our proof is by induction on $n$. The case $n=1$ is trivial. To
  partition $\mbf{2}^{n+1}-\empset$, start with the interval
  consisting of all subsets of $[n+1]$ containing $n+1$. This is an
  interval with a singleton as its minimal element. Removing it from
  $\mbf{2}^{n+1}-\empset$ leaves $\mbf{2}^n-\empset$, which we can
  partition into intervals with singletons as minimal elements by
  induction.
\end{proof}

\begin{thm}\label{thm:min-counterex}
  Let $K$ be a field and let $S = K[x_1,\dots,x_n]$. If there exists a
  squarefree monomial ideal $I\subseteq S$ such that $\sdepth I \leq
  \sdepth S/I$, then there exists an integer $n'\leq n$, a subset of indeterminates
  $\set{y_1,\dots,y_{n'}}\subseteq \set{x_1,\dots,x_n}$, and a
    squarefree monomial ideal $I'\subseteq K[y_1,\dots,y_{n'}]$ having
    \SR{} simplicial complex $\Delta$ satisfying
  \begin{enumerate}
  \item $\Delta$ is pure with
    $\dim(\Delta)+1= \sdepth S/I'\geq \sdepth I'$,
  \item the union of the facets of $\Delta$ is $[n']$, and
  \item the intersection of the facets of $\Delta$ is empty.
  \end{enumerate}

\end{thm}



\begin{proof}
  Let $I\subseteq S$ be a squarefree monomial ideal such that
  $k=\sdepth(S/I)\geq \sdepth I$. We also assume that amongst all such
  ideals, we have chosen $S$ so that $n$ is minimal.

  We begin by proving the first statement by showing how to construct
  such a simplicial complex starting with a squarefree monomial
  ideal for which $\sdepth I\leq \sdepth S/I$.  We assume that
  $P_{S/I}$ contains a set $X$ of size at least $k+1$, taking $|X|$ to
  be as large as possible. Let $\mc{P}_I$ be an optimal Stanley
  partition of $P_I$ and $\mc{P}_{S/I}$ an optimal Stanley partition
  of $P_{S/I}$. Consider the interval $I_X$ of $\mc{P}_{S/I}$
  containing $X$. By the maximality of $|X|$, $X$ is the upper bound
  of $I_X$. If $I_X$ is trivial, simply move $X$ from $P_{S/I}$ to
  $P_I$ to form $P_{S/J}$ and $P_J$ for the appropriate squarefree
  monomial $J$. Otherwise, notice that $I_X-X$ is isomorphic to the
  dual of $\mbf{2}^{|X|-1}-\empset$, so
  Proposition~\ref{prop:singleton-part} implies that there is a
  partition of $I_X-X$ into intervals in which each interval's upper
  bound has size $|X|-1$. We form $P_J$ from $P_I$ by adding $X$ to
  $P_I$ (and to $\mc{P}_I$ as a trivial interval in $\mc{P}_J$) and
  remove $X$ from $P_{S/I}$ to form $P_{S/J}$. 

  Replacing $I_X$ in $\mc{P}_{S/I}$ by the interval partition provided
  by Proposition~\ref{prop:singleton-part} gives a partition
  $\mc{P}_{S/J}$ that still has Stanley depth at least $k$ since the
  maximal element of each interval added has size at least $k$. We
  thus must show that $\sdepth J\leq k$ to retain the desired
  inequality.  Suppose that $\sdepth J > k$ and let $\mathcal{Q}_J$ be
  a witnessing partition.  Because $X$ was maximal in $P_{S/I}$, we
  know that $X$ is minimal in $P_J$. In particular there is an
  interval $[X,X'] \in \mathcal{Q}_J$ with $X\neq X'$ because
  otherwise $\mc{Q}_J - [X,X]$ would be a partition of $I$ having
  greater Stanley depth than $\mc{P}_I$. Let $\mathcal{P}_X$ be a the
  partition of $[X,X'] - X$ resulting from applying Proposition
  \ref{prop:singleton-part} to $[X,X'] - X$.  Let
  $\mc{Q}_I=\left(\mathcal{P}_J \cup \mathcal{P}_X\right) -
  \set{[X,X']}$. Notice that $\mc{Q}_I$ is a partition of $P_I$ in
  which the maximal element of every interval has size at least
  $\min\set{\sdepth J, |X|+1} > k$. Therefore, $\sdepth I > k$, which
  contradicts our assumption that $\sdepth I \leq \sdepth S/I = k$.

  If $P_{S/J}$ still contains a set of size greater than $k$, this
  process can be repeated until $P_{S/J}$ is a pure simplicial
  complex%
  , and the inequality on Stanley depth will be preserved throughout.



  We now prove the second property. Let $\mc{A}_I$ be the minimal
  elements of $P_I$ and let $\mc{A}_{S/I}$ be the maximal elements of
  $P_{S/I}$ (the facets of the \SR{} complex of $I$). Assume without
  loss of generality that there is no $A\in \mc{A}_{S/I}$ such that
  $n\in A$.  Let $\mc{A}'_I = \mc{A}\cap\mbf{2}^{[n-1]}$. Then
  $D[\mc{A}_{S/I}]\cup U[\mc{A}'_I]=\mbf{2}^{[n-1]}$, so by the
  minimality of $n$,
  $\sdepth U[\mc{A}'_I] > \sdepth D[\mc{A}_{S/I}]$ is witnessed by
  partitions $\mc{P}'_I$ and $\mc{P}_{S/I}$. But now
  $\mc{P}_I = \mc{P}'_I\cup \set{[\set{n},[n]]}$ is a partition of
  $U[\mc{A}_I]$ with Stanley depth equal to $\sdepth \mc{P}'_I$.
  Therefore, $\sdepth I > \sdepth S/I$, contradicting that $I$ and
  $S/I$ are a counterexample to the conjecture.


  We now proceed to prove the third property holds as well. Since we
  have already proved the first property, we may assume that
  $\Delta_I$ is pure with all facets of dimension $k-1$ where
  $k=\sdepth S/I$. Further assume that some element of $[n]$, without
  loss of generality we will say it is $n$, appears in every facet of
  $\Delta_I$. Notice that since $n$ is in every facet of
  $\Delta_I = P_{S/I}$ and this poset is a down set in $\mbf{2}^n$,
  $P_{S/I}$ can be viewed as the disjoint union of two isomorphic
  subposets of $P_{S/I}$---namely, all the subsets containing $n$ and
  all those not containing $n$. Let $P'_{S/I}$ be the subposet of
  $P_{S/I}$ consisting only of subsets of $[n-1]$. Since $P'_{S/I}$ is
  a down set of $\mbf{2}^{n-1}$, its complement $P'_I$ is an up set
  corresponding to some monomial ideal $J$ and $P'_{S/I} = P_{S'/J}$
  where $S'= K[x_1,\dots,x_{n-1}]$. By minimality, we know that
  $\sdepth J > \sdepth S'/J$ is witnessed by optimal Stanley
  partitions $\mc{P}_J$ and $\mc{P}_{S'/J}$. Now notice that adding
  $n$ to the upper bound of each interval of $\mc{P}_{S'/J}$ gives a
  Stanley partition of $P_{S/I}$ witnessing $\sdepth S/I = k$.
  Similarly, adding $n$ to the upper bounds of the intervals of
  $\mc{P}_J$ gives a Stanley partition of $P_I$ witnessing
  $\sdepth I\geq 1+\sdepth J \geq k+1$. This contradicts that $I$ and
  $S/I$ were a counterexample.
\end{proof}

Since satisfying Theorem~\ref{thm:min-counterex} when $\Delta_I$ is a
pure simplicial complex of dimension $n-2$ requires that
$P_I = \set{[n]}$, we have the following corollary.

\begin{cor}\label{cor:n-1}
  Let $I\subseteq S = K[x_1,\dots,x_n]$ be a squarefree monomial
  ideal. If the facets of $\Delta_I$ all have size $n-1$, then
  $\sdepth I > \sdepth S/I$.
\end{cor}


We conclude this section with a lemma that looks at what happens when $P_I$
consists only of sets of size at least $k$.
\begin{lem}
  \label{lem:cover-k-1} Let $I\subseteq S = K[x_1,\dots,x_n]$ be a
  squarefree monomial ideal. If $P_{S/I}$ contains all $(k-1)$-sets and
  $k\leq (n-1)/2$, then $\sdepth(I) > k$.
\end{lem}

\begin{proof}
  To show that $\sdepth
  I>k$, it suffices to find a complete matching from the $k$-sets in $P_I$ to
  the $(k+1)$-sets. Looking only at the $k$- and $(k+1)$-sets in $P_I$
  as a bipartite graph, we see that the partite set consisting of the
  $k$-sets is $(n-k)$-regular. The vertices in the partite set
  consisting of the $(k+1)$-sets all have degree at most $k+1$. Since
  $k\leq (n-1)/2$, $k+1\leq n-k$. Thus, by a corollary of Hall's
  Theorem, there is a complete matching from the $k$-sets to the
  $(k+1)$-sets. Use the edges of the matching to define intervals in a
  Stanley partition of $P_I$ with trivial intervals used to cover any
  set not in one of the matching intervals.
\end{proof}

\section{Combinatorial criteria}
\label{sec:comb-crit}
In this section, we continue to assume that $S=K[x_1,\dots,x_n]$ and
$I$ is a squarefree monomial ideal of $S$. We consider the posets
$P_I$ and $P_{S/I}$ as complementary subposets of $\mbf{2}^n$. If we
are focusing on proving that $\sdepth I> \sdepth S/I$,
Theorem~\ref{thm:min-counterex} allows us to restrict our attention to the
case where every maximal element of $P_{S/I}$ is a $k$-set where
$k=\sdepth S/I$. In this section, we establish some necessary
conditions on $P_{S/I}$ for this to happen. Although this conjecture
is our motivation, the conditions discussed here apply to a broader
class of posets and will be established in the fullest context.

The first criterion we establish is what we call the
\emph{combinatorial criterion}. Assuming $J\subseteq I$ are squarefree
monomial ideals of $S$, there exist antichains
$\mc{A},\mc{B}\subseteq\mbf{2}^n$ such that the poset $P_{I/J}$ is the
intersection of $U[\mc{A}]$ and $D[\mc{B}]$. In particular, $\mc{A}$
is the antichain of minimal elements of $P_{I/J}$ and $\mc{B}$ is the
antichain of maximal elements. If we wish to show that
$\sdepth I/J\geq k$, we must find a partition $\mc{P}_{I/J}$ of
$P_{I/J}$ into intervals so that each interval's upper bound has size
at least $k$. In fact, it suffices by a straightforward generalization
of property (1) of Theorem~\ref{thm:min-counterex} to put every set of
size at most $k$ into an interval with upper bound of size $k$,
leaving the remaining sets in trivial intervals.

For ease of explanation, let us temporarily focus on partitioning
$P_{S/I}$ where the maximal elements are all of size $k$. This
naturally requires an interval from the empty set to a $k$-set be
included in $\mc{P}_{S/I}$. Therefore, $P_{S/I}$ must include at least
$k$ singletons. Proceeding now to consider the number of $2$-sets in
$P_{S/I}$, we notice that the first interval requires at least
$\binom{k}{2}$ of them and that each singleton not covered by the
first interval will be the minimal element of an interval in
$\mc{P}_{S/I}$ covering $k-1=\binom{k-1}{1}$ $2$-sets. Thus, we have a
lower bound on the number of $2$-sets in $P_{S/I}$. If $P_{S/I}$ does
not have this number of $2$-sets, we say that $P_{S/I}$ \emph{fails to
satisfy the combinatorial criterion} (and we must necessarily have
$\sdepth S/I < k$).

Instead of introducing cumbersome notation to give equations that must
have integer solutions if the combinatorial criterion is satisfied, we
instead describe the process of verifying the combinatorial criterion
for a poset $P_{I/J}$. We begin with a vector
$\mbf{a}=(a_0,a_1,\dots,a_k)$ where $a_i$ is the number of $i$-sets in
$P_{I/J}$ for $i=0,\dots,k$. (If $P_{I/J}$ is considered as a
(relative) simplicial complex, then this is its $f$-vector, truncated
after counting the $(k-1)$-dimensional simplicies.) As we proceed,
$\mbf{a}$ will track the number of $i$-sets not already in an interval
of an interval partition in which all maximal elements have size
$i$. (Note that we do not actually construct the partition; we only
consider what structure it \emph{must} have if the upper bound of each
interval has size $k$.)  Supposing that all sets of size less than $i$
have been successfully covered, we know that there are $a_i$ remaining
sets of size $i$ to cover. If we are to successfully complete the
partition, each of these is the minimal element of an interval with
maximal element of size $k$, so each interval consumes
$\binom{k-i}{j-i}$ sets of size $j=i,\dots,k$. The vector $\mbf{a}$ is
then updated by decreasing $a_j$ by $a_i\binom{k-i}{j-i}$ for
$j=i,\dots,k$. If coordinate $a_j$ of $\mbf{a}$ is now negative for
some $j$, we know that there were insufficiently many $j$-sets for a
partition witnessing Stanley depth $k$ to exist. In this case, we say
that the \emph{combinatorial criterion is violated}. If we can repeat
this process up to $i=k$ without ever obtaining a negative entry in
$\mbf{a}$, then the \emph{combinatorial criterion is satisfied} and we
know that there are enough sets of each size in $P_{I/J}$ to support a
partition of the type sought. We note that satisfying the
combinatorial criterion is equivalent to the $h$-vector of $P_{S/I}$
(when it can be viewed as a pure simplicial complex) being nonnegative
by Proposition III.2.3 of Stanley's monograph
\cite{stanley:comb-comm-alg}.

Unfortunately, while the combinatorial criterion is necessary for
$P_{S/I}$ to have Stanley depth $k$, it is not sufficient. To see why,
suppose that $P_{S/I}$ is the down set of the antichain
$\set{123,124,125,134,345,234}$ in $\mbf{2}^5$. The interval with the
empty set as its minimal element covers three singletons and three
$2$-sets. This leaves two singletons, seven $2$-sets, and five
$3$-sets uncovered. Intervals beginning at the remaining two
singletons cover four more $2$-sets (two each) and two more
$3$-sets. Now $\mbf{a} = (0,0,3,3)$. Covering the remaining $2$-sets
requires three $3$-sets, reducing $\mbf{a}$ to the zero vector. Thus,
this poset does satisfy the combinatorial criterion. However, we will
now see that $\sdepth P_{S/I} <3$. This is because the strict up set
of $\set{5}$ in $P_{S/I}$ is $U=\set{15,25,35,45,125,345}$. Whatever
interval covers $\set{5}$ will consume two $2$-sets from $U$ and one
$3$-set, leaving behind two $2$-sets and one $3$-set. There is no way
for a single interval not containing $\set{5}$ to cover these three
remaining sets, so a Stanley partition of $P_{S/I}$ must contain an
interval with maximal element of size less than $3$.

The way in which the preceding example failed to have Stanley depth
$3$ is instructive in providing a stronger version of the
combinatorial criterion. Since it is not sufficient to have enough
sets of each size in $P_{I/J}$, we add the additional requirement that
for each set $A$ in $P_{I/J}$, the subposet $U[A]\cap P_{I/J}$ must
satisfy the combinatorial criterion. If $P_{I/J}$ satisfies this more
stringent requirement, we say that it satisfies the \emph{strong
  combinatorial criterion}. The example given above violates the
strong combinatorial criterion because the combinatorial criterion is
not satisfied for the up set of $\set{5}$. To see that the strong
combinatorial criterion is necessary for $P_{I/J}$ to have Stanley
depth $k$, we consider the restriction of an optimal Stanley partition
$\mc{P}_{I/J}$ of $P_{I/J}$ to the closed up set $U$ (inside
$P_{I/J}$) of some set $A\in P_{I/J}$. (Form this restriction by
intersecting each interval of $\mc{P}_{I/J}$ with $U$ and discarding
empty intervals that result.) Notice that $U$ has a single minimal
element $A$ and each maximal element of $U$ is a superset of
$A$. Thus, $U$ is isomorphic to $P_{S'/I'}$ for some ring $S' = K[T]$
with $T\subset\set{x_1,\dots,x_n}$ and squarefree monomial ideal
$I'\subseteq S'$. Since the restriction of $\mc{P}_{I/J}$ to $U$ gives
an interval partition of $U$ in which the upper bound of each interval
has size $k$, we know from above that $U$ must satisfy the
combinatorial criterion. From the simplicial complex perspective,
satisfying the strong combinatorial criterion is equivalent to the
$h$-vector of every link being nonnegative.

Because the strong combinatorial criterion can be checked by computer
via simple counting methods, it is much more efficient than verifying
the existence of a particular Stanley partition by brute
force. However, we also note that a very large fraction of squarefree
monomial ideals for which the maximal elements of $P_{S/I}$ are all of
the same size satisfy the strong combinatorial criterion for $k\leq
n/2$ and for $k>n/2$, the majority violate the strong combinatorial
criterion. See Tables~\ref{tab:n-6} and \ref{tab:n-7} in
Section~\ref{sec:compute} for relevant enumerations.

At one point, computer investigations emboldened us to believe that if
$\Delta$ is a pure simplicial complex of dimension $k-1$, then
$\sdepth(S/I_\Delta) = k$ if and only if $\Delta$ satisfies the
strong combinatorial criterion. However, the counterexample $C_3$ of
Duval et al. in \cite[Theorem 3.5]{duval:sdepth-counterex} has $n=16$,
is Cohen-Macaulay (hence pure) of dimension $3$, and satisfies the
strong combinatorial criterion. However, the Stanley depth of $S/I_{C_3}$ is
only $3$ for that example. 
It would be interesting to determine if every pure $k$-dimensional
simplicial complex $\Delta$ with $\sdepth(S/I_\Delta) < k+1$ contains
an antichain that provides a combinatorial witness to the fact that
the Stanley depth is not $k+1$.



\section{Splitting}\label{sec:splits}

Our final reduction result focuses on the scenario where $P_{S/I}$ can
be split into two pieces, one consisting of all sets containing an
element $x\in[n]$ and the other consisting of all sets not containing
$x$, in a particular way that allows for a Stanley partition to be
constructed. This technique will be useful in computational results
that come later. In what follows, we will refer to an antichain in the
subset lattice satisfying the strong combinatorial criterion, by which we
mean that its closed downset satisfies the strong combinatorial criterion.

\begin{defn}\label{defn:splits}
  Let $\mc{A}$ be an antichain in $\mbf{2}^{[n]}$ and $x\in [n]$. Let
  $\mc{A}_x = \set{A\in\mc{A}\mid x\in A}$ and $\mc{A}'_x =
  \set{A-x\mid A\in\mc{A}_x}$. We say that $\mc{A}$ \emph{splits over
    $x$} provided that
  \begin{enumerate}[(i)]
  \item for all $S\in \mc{A}'_x$, there exists $T\in \mc{A}-\mc{A}_x$
    such that $S\subseteq T$ and
  \item $\sdepth(D[\mc{A} - \mc{A}_x])\geq \sdepth(D[\mc{A}])$.
  \end{enumerate}
  The antichain $\mc{A}$ \emph{splits} provided that there exists
  $x\in [n]$ such that $\mc{A}$ splits over $x$.
\end{defn}

To illustrate this definition, we consider the following example.

\begin{examp}\label{ex:splits}
  Let $\mc{A}=\set{  \set{1, 2, 5}, \set{ 2, 4,6}, \set{1, 2, 3},
    \set{2, 4, 5}, \set{3, 4, 5}, \set{2, 3, 5}}$. We see immediately
  that $\mc{A}$ splits over $6$, since $\mc{A}_6 = \set{\set{2,4,6}}$
  and $\mc{A}'_6 = \set{2,4}\subset \set{2,4,5}$. (Verification of the
  combinatorial criterion requirements is straightforward but
  tedious.) Similarly, $\mc{A}$ splits over $1$ and $3$. However,
  $\mc{A}$ does not split over $2$ since $\mc{A}-\mc{A}_2 =
  \set{\set{3,4,5}}$ and $\set{1,5}\in\mc{A}'_2$ but
  $\set{1,5}\not\subseteq \set{3,4,5}$.
\end{examp}

We now show that the search for a counterexample to the conjecture can
be restricted to antichains that do not split.

\begin{lem}\label{lem:splits}
  Let $I\subseteq S = K[x_1,\dots,x_n]$ be a squarefree monomial
  ideal. If $\sdepth I\leq \sdepth S/I$ and the antichain of maximal
  elements of $P_{S/I}$ splits, then there exists a squarefree
  monomial ideal $J\subseteq S' = K[y_1,\dots,y_{n'}]$ with $n'\leq n$ and
  $\set{y_1,\dots,y_{n'}}\subseteq \set{x_1,\dots,x_n}$ such that
    $\sdepth J\leq \sdepth S'/J$ and the antichain of maximal elements
    of $P_{S'/J}$ does not split.
\end{lem}

\begin{proof}

  Suppose that $I$ and $S/I$ are minimal (in terms of $n$) such that
  $\sdepth I\leq \sdepth S/I$ and the antichain $\mc{A}$ of maximal
  elements of $P_{S/I}$ splits. Without loss of generality, we assume
  that $\mc{A}$ splits over $n$. By Theorem~\ref{thm:min-counterex},
  we may assume that every element of $\mc{A}$ is a $k$-set and
  $\sdepth(P_{S/I}) = k$. (Showing that the reductions of
  Theorem~\ref{thm:min-counterex} preserve splitting is a
  straightforward exercise.) Let
  $\mc{A}_n = \set{A\in\mc{A}\mid n\in A}$ and
  $\mc{B} = \mc{A} - \mc{A}_n$. Now $\mc{B}$ is an antichain of
  $k$-sets in $\mbf{2}^{[n-1]}$, so by minimality there is a partition
  $\mc{P}_{\mc{B}}$ witnessing
  $\sdepth(\mbf{2}^{[n-1]}-D[\mc{B}])> \sdepth(D[\mc{B}]) \geq k$. Furthermore, the first
  condition in the definition of what it means for $\mc{A}$ to split
  over $n$ ensures that no set covered by $\mc{P}_{\mc{B}}$ belongs to
  $D[\mc{A}]$. Similarly, $\mc{A}'_n=\set{A-n\mid A\in\mc{A}_n}$ is an
  antichain of $(k-1)$-sets in $\mbf{2}^{[n-1]}$, and there exists a
  partition $\mc{P}_{\mc{A}'_n}$ witnessing
  $\sdepth(\mbf{2}^{[n-1]} - D[\mc{A}'_n]) \geq k$. Form a new
  partition
  \[\mc{P}_n = \set{[S\cup\set{n},T\cup\set{n}]\mid [S,T]\in
    \mc{P}_{\mc{A}'_n}}. \]
  Let $\mc{D} = \set{S\cup\set{n}\mid S\in D[\mc{A}'_n]}$. Then
  $\mc{D}\cup\mc{P}_n = [\set{n},[n]]$ and
  $\mc{P}_{\mc{B}}\cup D[\mc{B}] = [\empset,[n-1]]$. Thus
  $\mc{P}_\mc{B}\cup \mc{P}_n$ is a partition of
  $\mbf{2}^{[n]} - D[\mc{A}]$ witnessing $\sdepth(I)\geq k+1$. But now
  we have
  $\sdepth(S/I)\leq k$, contradicting that $\sdepth I\leq \sdepth S/I$.
\end{proof}

\section{The $n=4$ and $n=5$ squarefree cases}

In this section, we address the conjecture that $\sdepth I > \sdepth
S/I$ for the cases where $S$ is the polynomial ring over a field $K$
in four or five indeterminates.  The first proof is immediate from our
reduction results, but we identify which result takes care of each of
the various cases.

\begin{thm}
  Let $S=K[x_1,\dots,x_4]$. If $I$ is a squarefree monomial ideal
  of $S$, then $\sdepth I > \sdepth S/I$.\label{thm:n4-sqfree}
\end{thm}

\begin{proof}
  We consider the posets $P_I$ and $P_{S/I}$ as complementary subposets of
  $\mbf{2}^4$. Let $\mc{A}_I$ be the minimal elements of $P_I$ and let
  $\mc{A}_{S/I}$ be the maximal elements of $P_{S/I}$.  By
  Theorem~\ref{thm:min-counterex}, we may assume that all elements of
  $\mc{A}_{S/I}$ have size $k$. The cases where $k=0$ and $k=4$ are
  trivial.

  If $k=1$, then by Theorem~\ref{thm:min-counterex}, a minimal
  counterexample to the theorem would require that $\mc{A}_{S/I}$
  consists precisely of the four singleton subsets of $[4]$. But then
  $\mc{A}_I$ contains only elements of size $2$, and the inequality of
  the theorem holds trivially.

  If $k=2$, it suffices to notice that $P_{S/I}$ must contain the
  four singleton subsets of $[4]$ by Theorem~\ref{thm:min-counterex}. But
  then all $(k-1)$-sets are contained in $P_{S/I}$ and $k\leq n/2$, so
  the inequality holds by Lemma~\ref{lem:cover-k-1}.

  The case when $k=3$ follows from Corollary~\ref{cor:n-1}. 
\end{proof}

We now proceed to the $n=5$ case, which takes a more intricate
argument. We know from our earlier results that the facets of
$\Delta_I$ (equivalently, maximal elements of $P_{S/I}$) form a uniform
hypergraph on $[n]$. Thus, for ease of reading, most of the discussion
is phrased in the language of hypergraphs instead of monomial ideals,
simplicial complexes, or posets. Recall that a hypergraph is
\emph{$k$-uniform} if every hyperedge has size $k$. The \emph{degree} of a
vertex in a (hyper)graph is the number of edges containing that
vertex. A (hyper)graph is called \emph{$d$-regular} if every vertex
has degree $d$. The \emph{degree sequence} of a (hyper)graph is the
nonincreasing sequence of its vertex degrees. 

\begin{thm}
  \label{thm:n5-sqfree} Let $S=K[x_1,\dots,x_5]$. If $I$ is a
  squarefree monomial ideal of $S$, then $\sdepth I> \sdepth S/I$.
\end{thm}

\begin{proof}
  Similarly to the proof of Theorem~\ref{thm:n4-sqfree}, we consider
  the posets $P_I$ and $P_{S/I}$ as complementary subposets of
  $\mbf{2}^5$ with $\mc{A}_I$ the minimal elements of $P_I$ and
  $\mc{A}_{S/I}$ the maximal elements of
  $P_{S/I}$. Theorem~\ref{thm:min-counterex} allows us to assume that all
  elements of $\mc{A}_{S/I}$ have size $k$ for some $k$ with
  $0\leq k\leq 5$. The desired inequality is trivially true in the
  cases $k=0$ and $k=5$. In the case with $k=1$, the inequality also
  follows immediately from Theorem~\ref{thm:min-counterex} since either some
  element of $[5]$ does not appear in any set in $\mc{A}_{S/I}$ or all
  elements of $\mc{A}_I$ have size at least $2$.

  If $k=2$, we note that if every element of $[5]$ appears in some set
  in $\mc{A}_{S/I}$, then all singletons are in $P_{S/I}$. Since
  $k-1=1$ here, Lemma~\ref{lem:cover-k-1} yields the inequality. On
  the other hand, if some element of $[5]$ does not appear in any
  element of $\mc{A}_{S/I}$, then Theorem~\ref{thm:min-counterex} implies
  the desired inequality. For $k=4$, the result follows from
  Corollary~\ref{cor:n-1}.


  To address the remaining case where $k=3$, we will consider
  $\mc{A}_{S/I}$ to be a hypergraph $\mc{H} = ([5],\mc{E})$. We first
  note that $\mc{H}$ must have at least four hyperedges. Three
  hyperedges are necessary to satisfy conditions (2) and (3) of
  Theorem~\ref{thm:min-counterex}. The degree sequence of a three-edge
  $3$-uniform hypergraph consistent with
  Theorem~\ref{thm:min-counterex} must be $(2,2,2,2,1)$. Up to
  isomorphism, there is a unique $3$-uniform hypergraph with this
  degree sequence; its edge set is $\set{123,124,345}$, which does not
  satisfy the combinatorial criterion. Therefore, $\mc{H}$ has at
  least four hyperedges. 

  We now show that there is a set $\mc{S}\subseteq \mc{E}$ of four
  hyperedges giving a subhypergraph $\mc{H}'=([5],\mc{S})$ with degree
  sequence $(3,3,3,2,1)$. To do so, we consider the possible degree
  sequences for a four-edge $3$-uniform subhypergraph $\mc{H}'$ of
  $\mc{H}$.
  \begin{enumerate}
  \item If $(4,2,2,2,2)$ is the degree sequence of $\mc{H}'$, then one
    vertex $v$ is in all four hyperedges in $\mc{S}$. We know by
    Theorem~\ref{thm:min-counterex} that there is an edge $X\in \mc{E}$
    such that $v\nin X$. Deleting any hyperedge of $\mc{H}'$ yields a
    subhypergraph with degree sequence $(3,2,2,1,1)$. Adding $X$ to
    this subhypergraph then yields a subhypergraph that must have
    degree sequence $(3,3,3,2,1)$ or $(3,3,2,2,2)$. In the former
    case, we're done. The latter case will be addressed last.
  \item If $(3,3,3,3,0)$ is the degree sequence of $\mc{H}'$, then
    some vertex $v$ is not in any hyperedge in $\mc{S}$. However, by
    Theorem~\ref{thm:min-counterex}, some hyperedge $Y\in\mc{E}$ contains
    $v$. Replacing any hyperedge of $\mc{H}'$ with $Y$ yields
    progress, as the new subhypergraph must have degree sequence
    $(3,3,3,2,1)$, in which case we're done, or $(4,3,2,2,1)$, and we
    will consider this case shortly.
  \item The sequence $(4,3,3,1,1)$ is not the degree sequence of a
    four-edge hypergraph, since removing
    the vertex of degree four would leave a graph with degree sequence
    $(3,3,1,1)$, but there is no graph with this degree
    sequence. Therefore, no such hypergraph can exist.
  \item If $(4,3,2,2,1)$ is the degree sequence of $\mc{H}'$, then a
    vertex $v$ is in all four edges of $\mc{H}'$. Removing $v$ leaves
    a graph with degree sequence $(3,2,2,1)$. The only graph with this
    degree sequence is a triangle with a pendant edge attached to one
    of its vertices. Thus, $\mc{H}'$ must contain a hyperedge
    consisting of the vertex of degree $4$, the vertex of degree $3$,
    and a vertex of degree $2$. Replacing this hyperedge with a
    hyperedge $X\in \mc{E}$ such that $v\nin X$, which must exist
    since $v$ has degree $4$ in $\mc{H}'$ but cannot be in all
    hyperedges of $\mc{H}$, yields a new subhypergraph that has degree
    sequence $(3,3,3,2,1)$ or $(3,3,2,2,2)$.
  \item It remains only to consider when $\mc{H}'$ has degree sequence
    $(3,3,2,2,2)$. In this case, there are two possibilities. The
    easier to resolve occurs when the three vertices of degree $2$ all
    appear together in a hyperedge of $\mc{H}'$ and the other three
    hyperedges contain the two vertices of degree $3$ and a vertex of
    degree $2$. Let $v$ be a vertex of degree $2$ in $\mc{H}'$. Since
    the two hyperedges of $\mc{H}'$ containing $v$ intersect only in
    $v$, we know that $\mc{E}$ has another hyperedge $X=\set{v,u,w}$
    not in $\mc{H}'$ because the strong combinatorial criterion must
    be satisfied (particularly, the combinatorial criterion for $U[v]$). Notice that without loss of generality, $u$ has
    degree $3$ in $\mc{H}'$ and $w$ has degree $2$ in
    $\mc{H}'$. Letting $x$ be the other vertex with degree $2$ in
    $\mc{H}'$ and $y$ the other vertex of degree $3$ in $\mc{H}'$, we
    replace the hyperedge $\set{u,x,y}$ by $X$ to form a new
    subhypergraph of $\mc{H}$ with degree sequence $(3,3,3,2,1)$ as
    desired.

    The other possible way for $\mc{H}'$ to have a degree sequence of
    $(3,3,2,2,2)$ is depicted in Figure~\ref{fig:33222-hypergraph}. We
    notice $\mc{H}$ must have more hyperedges than $\mc{H}'$, as the strong
    combinatorial criterion is violated by the up set of $x$ if these
    are all the hyperedges of $\mc{H}$.
    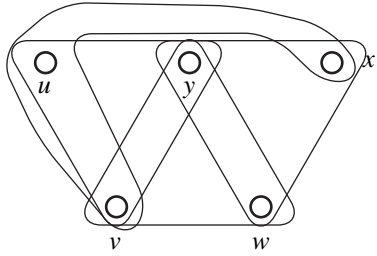
\begin{figure}[h]
      \centering
      \begin{overpic}{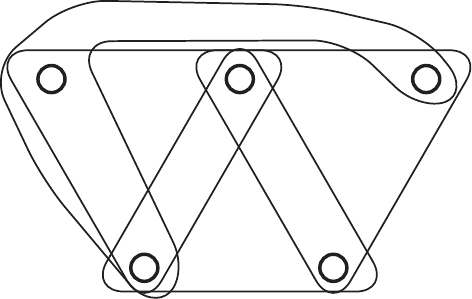}
        \put(135,62){$x$}
        \put(67,52){$y$}
        \put(12,52){$u$}
        \put(39,-7){$v$}
        \put(93,-7){$w$}
      \end{overpic}
      \caption{The second subhypergraph with degree sequence $(3,3,2,2,2)$}
      \label{fig:33222-hypergraph}
    \end{figure}
    There are four possibilities for the hyperedge of $\mc{H}$ that does
    not appear in $\mc{H}'$ but must exist to ensure the strong
    combinatorial criterion is satisfied. They are $\set{x,v,w}$,
    $\set{x,y,v}$, $\set{x,y,u}$, and $\set{x,w,u}$. In the first
    three cases, it is possible to add the hyperedge to $\mc{H}'$ and
    remove another edge to obtain a new subhypergraph with degree
    sequence $(3,3,3,2,1)$. If $\set{x,v,w}$ is added, then
    $\set{x,u,v}$ should be removed. For both $\set{x,y,v}$ and
    $\set{x,y,u}$, the hyperedge to remove is $\set{y,v,w}$. The case
    where $\set{x,w,u}$ is the hyperedge available to add requires a
    bit more work. Adding $\set{x,w,u}$ to $\mc{H}'$ yields the
    (unique) $3$-uniform, $3$-regular hypergraph on $5$
    vertices. However, this hypergraph cannot be $\mc{H}$, as the
    combinatorial criterion would be violated if it
    were, since the $f$-vector would be $(1,5,10,5)$. Without loss of
    generality, we may assume that $\mc{H}$ also contains the
    hyperedge $\set{x,v,w}$, so we remove $\set{x,u,v}$ and add
    $\set{x,v,w}$ to obtain a subhypergraph of $\mc{H}$ with degree
    sequence $(3,3,3,2,1)$ as desired.
  \end{enumerate}

  We now proceed to consider what the four hyperedges in $\mc{S}$, the
  set of hyperedges of $\mc{H}'$, can
  be. Our claim is that without loss of generality, $\mc{S} =
  \set{123,124,145,234}$. We begin by assuming that $3$ is the vertex
  contained in precisely two hyperedges $X,Y\in \mc{S}$. If $X\cap Y =
  \set{3}$, notice that the only way for the hyperedges in $\mc{S}$ to
  form a subhypergraph with degree sequence $(3,3,3,2,1)$ would be to
  have the three vertices of degree three appear together in a hyperedge
  and to use two copies of that hyperedge. Since this is impossible,
  $X\cap Y = \set{2,3}$ without loss of generality. In fact, we may
  assume $\set{123,234}\subset \mc{S}$. Notice now that $5$ can only
  appear in one element of $\mc{S}$, since if it appeared in three
  (the only other option), then each element of $\set{1,2,4}$ would
  appear in more than one hyperedge, contradicting the existence of a
  degree $1$ vertex in $\mc{H}'$. Furthermore, notice that $2$ and $5$
  cannot appear in the same hyperedge, as if they did, then vertices
  $2$, $3$, and $5$ would all have the required degree, preventing the
  addition of a fourth hyperedge. Therefore, $145\in\mc{S}$. At this
  point, vertices $3$ and $5$ have the desired degrees, but vertices
  $1$, $2$, and $4$ all have degree $2$. Therefore, $124\in \mc{S}$.

  If $\mc{A}_{S/I}=\mc{S}$, then the only $2$-sets of $\mbf{2}^5$ left
  uncovered by $P_{S/I}$ are $25$ and $35$. In this case, we may use
  the intervals $[25,1245]$ and $[35,1345]$ in the Stanley partition
  of $P_I$ (and use trivial intervals for the remainder). If
  $\mc{S}\subsetneq \mc{A}_{S/I}$, then these two intervals may be
  divided as needed to obtain a Stanley partition of $P_I$ witnessing
  $\sdepth I\geq 4$. 
\end{proof}

\section{Computational work for larger values of $n$}
\label{sec:compute}
Through the results of Section~\ref{sec:lemmas}, \ref{sec:comb-crit},
and \ref{sec:splits}, we have been able to
reduce the work of proving that $\sdepth(I) > \sdepth(S/I)$ for
$I\subset S$ a squarefree monomial ideal of $S=K[x_1,\dots,x_n]$ to a
size readily manageable by computer when $n=6$ and $n=7$. Taking, as
usual, $k$ to be the uniform size of the maximum elements of
$P_{S/I}$, the $k=0$ and $k=n$ cases are trivial. As with $n=5$, the
$k=1$ and $k=2$ cases are immediate consequences of
Theorem~\ref{thm:min-counterex} and Lemma~\ref{lem:cover-k-1}. The $k=n-1$ case
follows from Corollary~\ref{cor:n-1}.

We relied upon computer search for the $k=3$ and $k=4$ cases with
$n=6$ and $n=7$ and the $k=5$ case with $n=7$. We used McKay's
\texttt{nauty} package \cite{mckay:nauty} to generate the necessary
nonisomorphic hypergraphs on $6$ and $7$ vertices.

We developed a collection of Python routines to test the hypergraphs
for the various properties discussed in Sections~\ref{sec:lemmas}
through \ref{sec:splits}. Because of the repercussions for the degrees
of vertices in the hypergraph, a hypergraph failing to satisfy
properties (2) and (3) of Theorem~\ref{thm:min-counterex} was labelled
as having ``bad degree''. Hypergraphs for which the corresponding down
set failed the strong combinatorial criterion were not considered
further. In light of Lemma~\ref{lem:splits}, only antichains that did
not split were actually tested to find the Stanley depth of $P_I$. In
every case where this examination was conducted (by a nearly brute
force search of possible partitions), we confirmed that
$\sdepth(P_I)\geq k+1$. The counts of hypergraphs falling into each
category are shown in Tables~\ref{tab:n-6} and
\ref{tab:n-7}. Computations were done via distributed computing at
Washington and Lee University and the University of Louisville. All
source code as well as the partitions of hypergraphs into the
categories below are available for download from \url{https://github.com/mitchkeller/stanley-depth}

\begin{table}[h]
  \centering
\begin{tabular}{c|c||c|c|c|c}
  $k$& Total & Bad Degree& Fail SCC& Splits & $\sdepth(P_I)\geq k+1$\\\hline
  $3$ & 2136 & 57 & 527 & 1496 & 56\\
  $4$ & 156 & 35 & 55 & 66 & 0
\end{tabular}

\caption{Computational results for $n=6$}
\label{tab:n-6}
\end{table}

\begin{table}[h]
  \centering
\begin{tabular}{c|c||c|c|c|c}
  $k$& Total & Bad Degree& Fail SCC& Splits & $\sdepth(P_I)\geq k+1$\\\hline
  $3$ & 7013319& 2257& 888308& 5987476& 135278\\
  $4$ &7013319 & 2257& 4439735& 2383294& 188033 \\
  $5$ & 1043& 156& 589& 298& 0
\end{tabular}

\caption{Computational results for $n=7$}
\label{tab:n-7}
\end{table}

As a consequence of these computations and earlier results in this
paper, we have the following result
\begin{thm}
  \label{thm:n-leq-7} If $I$ is a squarefree monomial ideal of
  $S=K[x_1,\dots,x_n]$ with $n\leq 7$, then $\sdepth I > \sdepth S/I$.
\end{thm}

We also used Python code to explicitly calculate the Stanley depth
of $I$ and of $S/I$ to look for trends in the difference between these
values for different values of $k$. No clear trends emerged, but the
results of the computational work is reflected in Table~\ref{tab:gap}
and Figure~\ref{fig:gap}. It is worth noting that, given the counts for
bad degree from Table~\ref{tab:n-7}, there are situations
where there is not a natural reduction of the problem to a smaller
value of $n$ but the gap between $\sdepth(I)$ and $\sdepth(S/I)$ is
still greater than $1$. Furthermore, the instances where the gap
between the Stanley depth of $S/I$ and $I$ is precisely one include
instances that satisfy the strong combinatorial criterion as well as
instances that violate it.

\begin{table}[h]
  \centering
  \begin{subfigure}{0.4\textwidth}
    \centering
  \begin{tabular}{r|cc}
    &$4$&$5$\\\hline
    $1$ & $13$ &\\
    $2$ & $1026$ & $4$
  \end{tabular}
\caption{$k=2$}
\end{subfigure}
  \begin{subfigure}{0.4\textwidth}
    \centering
  \begin{tabular}{r|cc}
    &$4$&$5$\\\hline
    $2$ & $886423$ & $2424$\\
    $3$ & $4878319$ & $1246153$
  \end{tabular}
\caption{$k=3$}
\end{subfigure}

  \begin{subfigure}{0.40\textwidth}
    \centering
  \begin{tabular}{r|ccc}
    &$4$&$5$ & $6$\\\hline
    $3$ & $279$ & $4440053$ & \\
    $4$ &  & $2572970$ & $17$
  \end{tabular}
\caption{$k=4$}
\end{subfigure}
  \begin{subfigure}{0.4\textwidth}
    \centering
  \begin{tabular}{r|cc}
    &$5$&$6$ \\\hline
    $4$ & $282$ & $369$\\
    $5$ &  & $392$
  \end{tabular}
\caption{$k=5$}
\end{subfigure}

  \caption{Number of pure simplicial complexes $\Delta$ of dimension $k-1$ (for
    $S=K[x_1,\dots,x_7]$)
    with $\sdepth(S/I_\Delta)$ as given by row headings and $\sdepth(I_\Delta)$
    as given by column headings.}
  \label{tab:gap}
\end{table}

\begin{figure}[b]
  \centering
  \includegraphics[width=0.75\linewidth]{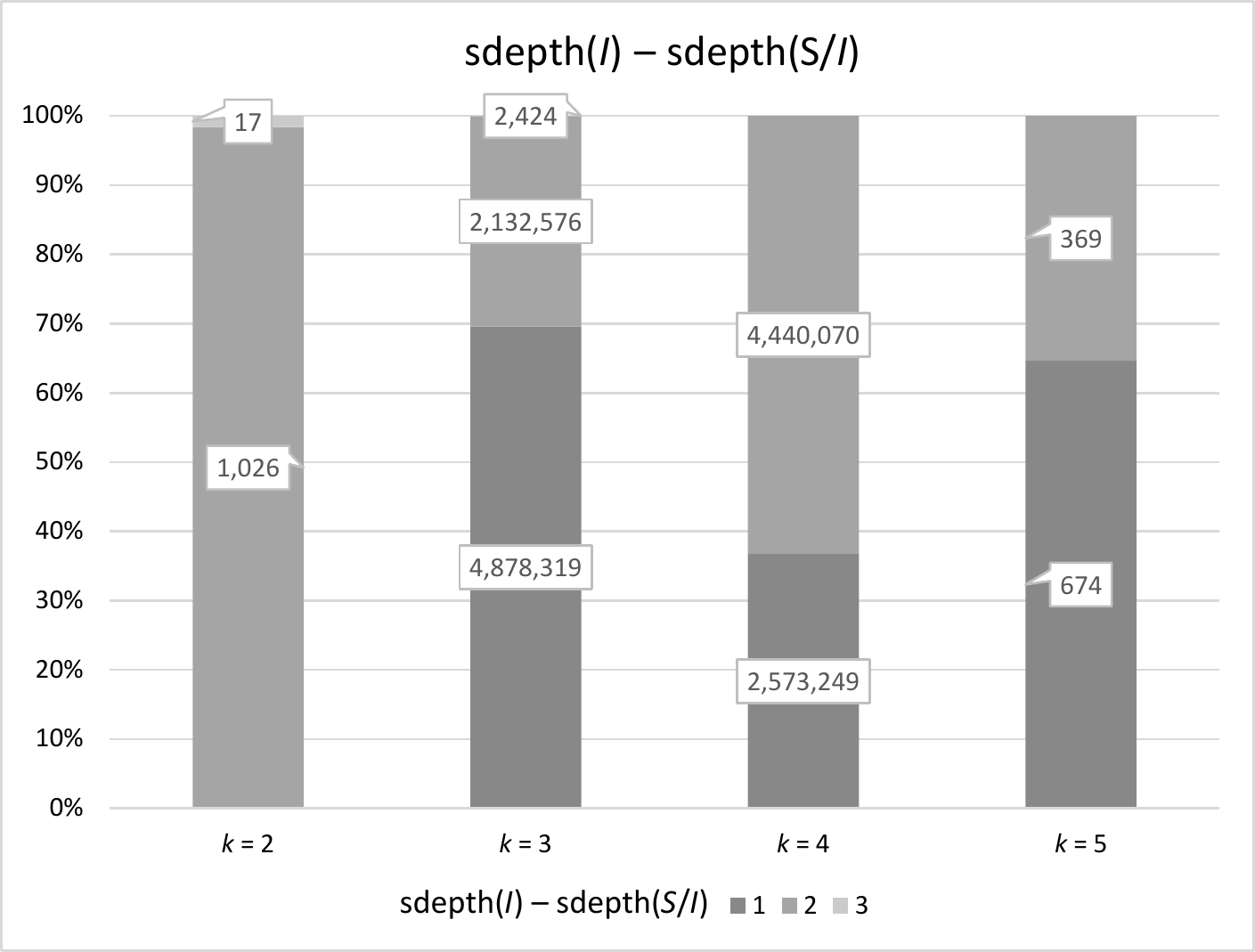}
  \caption{Quantifying the gap between $\sdepth(S/I)$ and
    $\sdepth(I)$}
  \label{fig:gap}
\end{figure}

Attempts to extend the computational work to increase the range of $n$
for which Theorem~\ref{thm:n-leq-7} is valid appears infeasible at
this time. Resolving the case with $n=8$ and $k=3$ (which would allow
for concurrently resolving $k=5$) would require generating over $894$
billion hypergraphs, based on the enumeration of $k$-uniform
hypergraphs on $n$ vertices done by Qian 
\cite{qian:hypergraph-enum}. (This is in contrast to approximately
$3.5$ million hypergraphs generated for the entire $n=7$ case.)

In addition to the computations necessary to prove Theorem~\ref{thm:n-leq-7},
we also investigated the question raised by Duval et al.\ in
\cite{duval:sdepth-counterex} as to whether or not their
counterexample to the partitionability conjecture was the smallest
possible. We used our Python code to compute the Stanley depth, and
whenever that result was less than the size of the sets in the
antichain of maximal elements of $P_{S/I}$, we used SageMath
\cite{sagemath} to determine if $P_{S/I}$ (viewed as a simplicial
complex) was Cohen-Macaulay. As a consequence, we have Theorem~\ref{thm:partition-7}.

\begin{thm}\label{thm:partition-7}
  Let $K$ be a field, $S = K[x_1,\dots,x_n]$, and $I\subseteq S$ a
  squarefree monomial ideal. Let $\Delta_I$ be the \SR{}
  complex of $I$. If $n\leq 7$ and $\Delta_I$ is Cohen-Macaulay, then
  $\sdepth S/I = \dim \Delta_I+1$. Equivalently, $\Delta_I$ is
  partitionable. 
\end{thm}

As an extension of Theorem~\ref{thm:partition-7}, we have the
following:

\begin{cor}\label{cor:depth-sdepth-quotient}
    Let $K$ be a field, $S = K[x_1,\dots,x_n]$, and $I\subseteq S$ a
  squarefree monomial ideal. If $n\leq 7$, then $\depth S/I\leq
  \sdepth S/I$.
\end{cor}

This generalizes of a result of Popescu in
\cite{popescu:ineq-depth-sdepth}, where it was proved for $n\leq
5$. Corollary~\ref{cor:depth-sdepth-quotient} follows from
Theorem~\ref{thm:partition-7} by Corollary~37 of
\cite{herzog:sdepth-survey}, which reduces to the Cohen-Macaulay case
by taking skeleta in a way that does not increase the number of
vertices.

\subsection*{Acknowledgements} This work was conducted in part using
the resources of the University of Louisville's research computing
group and the Cardinal Research Cluster. The authors are grateful to
the anonymous referee whose suggestions added additional connections
to the existing literature.

\bibliographystyle{acm}
\bibliography{zotero}

\end{document}